\documentclass[11pt]{article}
\usepackage{amsmath,amssymb,amsthm}
\usepackage[margin=1in]{geometry}

\newtheorem{theorem}{Theorem}[section]
\newtheorem{lemma}[theorem]{Lemma}
\newtheorem{corollary}[theorem]{Corollary}
\theoremstyle{definition}
\newtheorem{definition}[theorem]{Definition}
\newtheorem{remark}[theorem]{Remark}

\newcommand{\Fstar}{F^{\!*}}

\title{Enumerating forcing and strongly forcing $(0,1)$-matrices}
\author{Lei Cao \and Jesse Geneson}
\date{}

\begin{document}
\maketitle

\begin{abstract}
Let $Q$ be a nonzero $s\times t$ $(0,1)$-pattern, and let $m\ge s$ and $n\ge t$. An
$m\times n$ matrix is
\emph{strongly $Q$-forcing} if every $1$-entry belongs to an $s\times t$ submatrix equal to
$Q$. Let $\Fstar(m,n,Q)$ count these matrices. Put $H=m-s+1$ and $W=n-t+1$. We prove
\[
\Fstar(m,n,Q)\ge 2^{HW}.
\]
Writing $r$ and $c$ for the numbers of nonzero rows and columns of $Q$, equality holds if and
only if
\[
(H=1\text{ or }r=1)\qquad\text{and}\qquad(W=1\text{ or }c=1).
\]
Thus the minimum over all nonzero $s\times t$ patterns is $2^{HW}$, attained exactly by
singleton patterns when $H,W>1$, and every fixed nonzero pattern has square growth rate $1$.
We also refine the count by weight. If $o(Q)$ is the number of $1$-entries of $Q$, then the
number of strongly $Q$-forcing matrices at the minimum positive weight $o(Q)$ is
$\binom{H+r-1}{r}\binom{W+c-1}{c}$; at every fixed density in $(0,1)$, the logarithmic growth
rate is the binary entropy when $m$ and $n$ are comparable. For ordinary forcing, where every
$s\times t$ submatrix contains the $1$-entries of $Q$ in their prescribed positions, let
$F(m,n,Q)$ be the number of forcing matrices and let $\mathfrak m(m,n,Q)$ be their minimum
weight. We prove
\[
F(m,n,Q)=2^{mn-\mathfrak m(m,n,Q)}
\quad\text{and}\quad
2^{\,mn-\mathfrak m(m,n,Q)+HW}
\le
F(m,n,Q)\Fstar(m,n,Q)
\le
2^{mn}.
\]
The lower product bound has the same equality cases as the
strong-forcing lower bound above, while the upper product bound is
attained exactly by singleton patterns. In particular, the product is
at least $2$, with equality exactly when $s=m$, $t=n$, and $Q$ is the
all-ones pattern.
\end{abstract}



\section{Introduction}\label{sec:intro}

For a fixed $(0,1)$-matrix $P$, the extremal function $\mathrm{ex}(n,P)$ is the maximum number of
$1$-entries in an $n\times n$ $(0,1)$-matrix that does not contain $P$. Here a matrix \emph{contains}
$P$ if it has a submatrix, chosen by an increasing set of rows and an increasing set of columns, with
a $1$ in every position where $P$ has a $1$; equivalently, if some submatrix can be turned into $P$ by
changing $1$'s to $0$'s. The study of $\mathrm{ex}(n,P)$ was initiated by Bienstock and
Gy\H{o}ri~\cite{BienstockGyori} and by F\"uredi and Hajnal~\cite{FurediHajnal}, the latter of whom
developed a Davenport--Schinzel theory of matrices and conjectured that $\mathrm{ex}(n,P)=O(n)$ for
every permutation matrix $P$. Klazar showed that this conjecture implies the Stanley--Wilf
conjecture on the number of pattern-avoiding
permutations~\cite{Klazar}, and Marcus and Tardos then proved the linear bound itself, settling both
conjectures at once~\cite{MarcusTardos}. Klazar and Marcus later extended the linear bound to ordered
hypergraphs and to higher dimensions~\cite{KlazarMarcus}, and Fox showed that the limiting constants
behind it are typically exponential rather than polynomial in the pattern size~\cite{Fox}.

Tardos studied small excluded submatrices \cite{Tardos}, while Keszegh and Pettie obtained bounds
for further linear and nonlinear patterns \cite{Keszegh,Pettie}. Janzer, Janzer, Magnan, and
Methuku proved tight general bounds for patterns with bounded row weight
\cite{JanzerJanzerMagnanMethuku}, and Pettie and Tardos constructed acyclic patterns whose
extremal functions exceed $n(\log n)^C$ for every fixed $C$ \cite{PettieTardos}. The theory also
extends to ordered graphs and higher-dimensional matrices
\cite{PachTardos,KlazarMarcus,GenesonTian,GenesonTsai}.
Extremal bounds have also been used to enumerate pattern-avoiding permutations. Brualdi and Cao
studied maximal pattern-avoiding matrices and bases of permutation matrices
\cite{BrualdiCaoAvoid}. Eu and Lee recently gave product formulas for maximal
identity-avoiding $(0,1)$-matrices and their symmetry classes \cite{EuLee}. These results concern
avoidance. Here we enumerate matrices satisfying the forcing conditions introduced by Cao and
Tsai.

After studying blockers of pattern-avoiding permutation matrices \cite{BrualdiCaoBlockers},
Brualdi and Cao~\cite{BrualdiCao} introduced $123$-forcing matrices: an $n\times n$
matrix is $123$-forcing if every permutation matrix beneath it contains $I_3$. They
characterized the $123$-forcing matrices with the fewest zeros. Cao and
Tsai~\cite{CT} defined two forcing notions for an arbitrary nonzero $s\times t$ pattern $Q$. An
$m\times n$ $(0,1)$-matrix $A$ is
\emph{$Q$-forcing} if every $s\times t$ submatrix of $A$ can be turned into $Q$ by changing some of its
$1$'s to $0$'s, equivalently if every $s\times t$ submatrix has a $1$ in each position where $Q$ does.
It is \emph{strongly $Q$-forcing} if every $1$-entry of $A$ lies in a copy of $Q$, where a copy of $Q$
reproduces both the $1$'s and the $0$'s of $Q$ on an increasing choice of $s$ rows and $t$ columns, and
the $1$-entry occupies a position of the copy at which $Q$ has a $1$.
The two notions are incomparable. For $Q=(1,0)$, the matrix $(0,0)$ is strongly
$Q$-forcing but not $Q$-forcing, while $(1,1)$ is $Q$-forcing but not strongly $Q$-forcing.

Cao and Tsai studied the extremal functions of both notions. They wrote
$\mathfrak m(m,n,Q)$ for the minimum number of $1$-entries in a $Q$-forcing matrix and
$M(m,n,Q)$ for the maximum number of $1$-entries in a strongly $Q$-forcing matrix. They gave
exact formulas for $\mathfrak m(m,n,Q)$ using corner functions and Young diagrams. They also
proved that $mn-M(m,n,Q)=O(m+n)$ for every fixed $Q$, determined $M(n,n,Q)$ for every $2\times2$ and
$3\times3$ permutation matrix, and conjectured a formula for $M(n,n,I_k)$.

We enumerate both forcing families. We are not aware of earlier enumerative results for either
notion. Write $F(m,n,Q)$ and $\Fstar(m,n,Q)$ for the numbers of ordinary and strongly $Q$-forcing
$m\times n$ matrices. Ordinary forcing is upward closed. In fact, the forcing matrices form the
principal up-set above the Cao--Tsai canonical minimum matrix, and therefore
\[
F(m,n,Q)=2^{mn-\mathfrak m(m,n,Q)}.
\]
Strong forcing need not be upward closed: adding a $1$ may create an entry that belongs to no
exact copy of $Q$.

Put
\[
H:=m-s+1,\qquad W:=n-t+1,\qquad
d_Q(m,n):=mn-\log_2\Fstar(m,n,Q).
\]
Our main theorem is the sharp bound
\[
\Fstar(m,n,Q)\ge2^{HW},\qquad
d_Q(m,n)\le mn-HW.
\]
The construction starts with an arbitrary $H\times W$ free block and surrounds it by a
scaffold modeled on $Q$. Scaffold arms corresponding to zero rows or columns of the free block
are pruned. Every remaining $1$ then belongs to a canonical copy of $Q$, so all $2^{HW}$ free
blocks give distinct strongly forcing matrices.

This bound is sharp for all $m,n,s,t$:
\[
\min_{\substack{Q\in\{0,1\}^{s\times t}\\Q\ne0}}\Fstar(m,n,Q)=2^{HW}.
\]
If $r(Q)$ and $c(Q)$ denote the numbers of nonzero rows and columns of $Q$, then equality in
the lower bound holds if and only if
\[
(H=1\text{ or }r(Q)=1)\quad\text{and}\quad
(W=1\text{ or }c(Q)=1).
\]
In particular, when $H,W>1$, the equality cases are exactly the singleton patterns. On the
diagonal, the lower bound also implies that
$n^{-2}\log_2\Fstar(n,n,Q)\to1$ for every fixed nonzero $Q$.

Ordinary and strong forcing also satisfy sharp two-sided product bounds.
Combining the ordinary count with the strong lower bound gives the
left-hand inequality below. For the right-hand inequality, every
$1$-entry of a strongly forcing matrix is eligible to occupy a
$1$-position of $Q$, so its support lies inside the
$\mathfrak m(m,n,Q)$ entries of the canonical minimum matrix. Hence
\[
2^{\,mn-\mathfrak m(m,n,Q)+HW}
\le
F(m,n,Q)\Fstar(m,n,Q)
\le
2^{mn}.
\]
Equality in the lower bound occurs precisely when
\[
\bigl(H=1\text{ or }r(Q)=1\bigr)
\qquad\text{and}\qquad
\bigl(W=1\text{ or }c(Q)=1\bigr),
\]
while equality in the upper bound occurs exactly for singleton
patterns. In particular, the product equals $2$ exactly when
$s=m$, $t=n$, and $Q$ is the all-ones pattern.

We also refine the strong count by weight. Let $o(Q)$ be the number of $1$-entries of $Q$.
There are no nonzero strongly forcing matrices of weight below $o(Q)$, and at the minimum
positive weight we prove the exact formula
\[
\Fstar(m,n,Q,o(Q))
=\binom{H+r(Q)-1}{r(Q)}\binom{W+c(Q)-1}{c(Q)}.
\]
If $w/(mn)\to\alpha\in(0,1)$ while $m/n$ stays bounded away from $0$ and $\infty$, then
\[
\frac{\log_2\Fstar(m,n,Q,w)}{mn}\longrightarrow
h(\alpha).
\]
We also obtain
exact formulas for all-ones column patterns, column patterns with one zero or with boundary
zero blocks, singleton patterns, and full-height all-ones patterns, and determine
$M(m,n,Q)=mn-z(Q)$ whenever $Q$ has an all-ones row and an all-ones column.

Sections~\ref{sec:construction} and~\ref{sec:mainproof} define the construction and prove
Theorem~\ref{thm:main}.
Section~\ref{sec:special} treats full scaffolds and maximum weight,
Section~\ref{sec:tightness} gives the exact families, Section~\ref{sec:weight} proves the
weight-refined results and the equality classification, and Section~\ref{sec:ordinary} proves
the ordinary count and the ordinary--strong product inequality.

\section{Definitions}

Fix integers $s,t\ge 1$. A \emph{pattern} is an $s\times t$ matrix
$Q=(Q[a,b])_{a\in[s],\,b\in[t]}$ with entries in $\{0,1\}$, where $[m]:=\{1,\dots,m\}$.
We assume throughout that $Q$ has at least one $1$-entry; the all-zero case is treated in
Remark~\ref{rem:nozero}. Throughout we use the convention that an integer interval
$\{a,\dots,b\}$ is \emph{empty} when $a>b$, and that a sum over an empty index set is zero.

\begin{definition}
Let $A$ be an $m\times n$ $(0,1)$-matrix. A \emph{copy of $Q$ in $A$} is a pair of strictly
increasing index sequences
\[
1\le r_1<r_2<\dots<r_s\le m,\qquad 1\le c_1<c_2<\dots<c_t\le n
\]
such that $A[r_a,c_b]=Q[a,b]$ for all $a\in[s]$, $b\in[t]$. Thus a copy reproduces both the
$1$'s and the $0$'s of $Q$. An entry $(i,j)$ of $A$ with $A[i,j]=1$ is \emph{witnessed} if there
is a copy $(r_\bullet,c_\bullet)$ of $Q$ in $A$ and a pair $(a,b)\in[s]\times[t]$ with
$Q[a,b]=1$, $r_a=i$, and $c_b=j$. The matrix $A$ is \emph{strongly $Q$-forcing} if every
$1$-entry of $A$ is witnessed; the all-zero matrix is strongly $Q$-forcing vacuously. Let
$\Fstar(m,n,Q)$ denote the number of $m\times n$ strongly $Q$-forcing matrices, and set
\[
d_Q(m,n):=mn-\log_2 \Fstar(m,n,Q).
\]
We write $\lambda_Q:=\lim_{n\to\infty} n^{-2}\log_2 \Fstar(n,n,Q)$ for the square growth rate,
when this limit exists; that it does is part of the assertion of Theorem~\ref{thm:main}. In
all asymptotic statements, the pattern $Q$ is held fixed.
Finally, $z(Q)$ denotes the number of $0$-entries of $Q$, we write $|A|_1$ for the number of
$1$-entries of a matrix $A$, and $M(m,n,Q)$ for the maximum of $|A|_1$ over all $m\times n$
strongly $Q$-forcing matrices $A$.
\end{definition}

\begin{theorem}\label{thm:main}
Let $Q$ be an $s\times t$ pattern with at least one $1$-entry, and write $H=m-s+1$ and $W=n-t+1$. Then
for all $m\ge s$ and $n\ge t$,
\begin{equation}\label{eq:mainbound}
\Fstar(m,n,Q)\ge 2^{HW}=2^{(m-s+1)(n-t+1)}.
\end{equation}
Equivalently,
\[
d_Q(m,n)\le mn-HW
=(s-1)n+(t-1)m-(s-1)(t-1).
\]
Consequently the square growth rate $\lambda_Q$ exists and equals $1$.
\end{theorem}

\begin{remark}\label{rem:nozero}
If $Q$ has no $1$-entry then no $1$-entry of any $A$ can be aligned to a $1$ of $Q$, so the only
strongly $Q$-forcing matrix is the all-zero matrix: $\Fstar(m,n,Q)=1$ and $\lambda_Q=0$. Thus the
hypothesis that $Q$ has at least one $1$-entry is necessary.
\end{remark}

\section{The construction}\label{sec:construction}

Fix a $1$-entry $(p,q)$ of $Q$, i.e.\ $Q[p,q]=1$, $p\in[s]$, $q\in[t]$. Fix $m\ge s$ and
$n\ge t$ and put
\[
H:=m-s+1\ge1,\qquad W:=n-t+1\ge1 .
\]
Partition the row index set $[m]$ into three consecutive intervals (empty ones allowed, per the
convention above)
\[
T^{\mathrm r}:=\{1,\dots,p-1\},\quad
\Phi^{\mathrm r}:=\{p,\dots,p+H-1\},\quad
B^{\mathrm r}:=\{p+H,\dots,m\},
\]
of sizes $p-1$, $H$, and $s-p$, whose sum is $H+s-1=m$. Call the elements
of $\Phi^{\mathrm r}$ the \emph{free rows} and those of $T^{\mathrm r}\cup B^{\mathrm r}$ the
\emph{scaffold rows}. Likewise partition the column set $[n]$:
\[
L^{\mathrm c}:=\{1,\dots,q-1\},\quad
\Phi^{\mathrm c}:=\{q,\dots,q+W-1\},\quad
R^{\mathrm c}:=\{q+W,\dots,n\},
\]
of sizes $q-1,\,W,\,t-q$ (again summing to $W+t-1=n$); the elements of $\Phi^{\mathrm c}$ are the
\emph{free columns}.

Assign to each scaffold row a \emph{role} $\rho(r)\in[s]\setminus\{p\}$ and to each scaffold
column a role $\sigma(c)\in[t]\setminus\{q\}$ by
\[
\rho(r)=\begin{cases} r, & r\in T^{\mathrm r},\\ r-H+1, & r\in B^{\mathrm r},\end{cases}
\qquad
\sigma(c)=\begin{cases} c, & c\in L^{\mathrm c},\\ c-W+1, & c\in R^{\mathrm c}.\end{cases}
\]
For $r\in T^{\mathrm r}$ this gives $\rho(r)\in\{1,\dots,p-1\}$; for $r\in B^{\mathrm r}$,
$\rho(p+H)=p+1,\dots,\rho(m)=s$, so $\rho(r)\in\{p+1,\dots,s\}$. Thus $\rho$ is a bijection from
the scaffold rows onto $[s]\setminus\{p\}$, and analogously $\sigma$ onto $[t]\setminus\{q\}$.

\paragraph{The matrices $A(B)$.} For a \emph{free block}
$B\in\{0,1\}^{\Phi^{\mathrm r}\times\Phi^{\mathrm c}}$ define the $m\times n$ matrix $A=A(B)$ by
\begin{equation}\label{eq:def}
A[r,c]=
\begin{cases}
B[r,c], & r\in\Phi^{\mathrm r},\ c\in\Phi^{\mathrm c}\quad(\text{free cell}),\\[2pt]
Q[p,\sigma(c)], & r\in\Phi^{\mathrm r},\ c\notin\Phi^{\mathrm c}\quad(\text{free row, scaffold col}),\\[2pt]
Q[\rho(r),q], & r\notin\Phi^{\mathrm r},\ c\in\Phi^{\mathrm c}\quad(\text{scaffold row, free col}),\\[2pt]
Q[\rho(r),\sigma(c)], & r\notin\Phi^{\mathrm r},\ c\notin\Phi^{\mathrm c}\quad(\text{scaffold cell}).
\end{cases}
\end{equation}
Cao and Tsai used row and column duplication in the proof of Lemma~3.1 of \cite{CT}.
The full scaffold $A(B)$ generalizes that construction. Their construction is the special
case obtained from their chosen anchor and the all-one block $B$. We allow the anchor to be
any $1$-entry and $B$ to be arbitrary; the pruned scaffold in
Section~\ref{sec:mainproof} deletes unsupported entries and makes all $2^{HW}$ free blocks
available for enumeration.
Thus the entries of $A(B)$ \emph{outside} the free block $\Phi^{\mathrm r}\times\Phi^{\mathrm c}$
are completely determined by $Q$ and the anchor $(p,q)$, while the entries \emph{inside} the free
block are the arbitrary bits of $B$. In particular the map $B\mapsto A(B)$ is injective, since $B$
is recovered as the restriction of $A(B)$ to $\Phi^{\mathrm r}\times\Phi^{\mathrm c}$. For brevity
we say a \emph{column of $B$} (resp.\ \emph{row of $B$}) is the restriction of $B$ to a single
free column $c\in\Phi^{\mathrm c}$ (resp.\ free row $r\in\Phi^{\mathrm r}$), and call it
\emph{nonzero} if it contains a $1$.

\section{Proof of Theorem~\ref{thm:main}}\label{sec:mainproof}

For a free row $i\in\Phi^{\mathrm r}$ write $R(i)$ for the sequence obtained by listing
$T^{\mathrm r}\cup\{i\}\cup B^{\mathrm r}$ in increasing order; for a free column
$j\in\Phi^{\mathrm c}$ write $C(j)$ for $L^{\mathrm c}\cup\{j\}\cup R^{\mathrm c}$ in increasing
order. Each is a strictly increasing sequence of length $s$ (resp.\ $t$), since it lists $s$
(resp.\ $t$) distinct indices.

\begin{lemma}\label{lem:order}
Let $i\in\Phi^{\mathrm r}$ and write $R(i)=(r_1<\dots<r_s)$. Then $r_p=i$, and for every
$a\ne p$ the entry $r_a$ is the scaffold row of role $a$, i.e.\ $\rho(r_a)=a$. The analogous
statement holds for $C(j)=(c_1<\dots<c_t)$: $c_q=j$, and $\sigma(c_b)=b$ for every $b\ne q$.
\end{lemma}

\begin{proof}
Every element of $T^{\mathrm r}=\{1,\dots,p-1\}$ is $<p\le i$, and every element of
$B^{\mathrm r}=\{p+H,\dots,m\}$ is $>p+H-1\ge i$ (as $i\le p+H-1$). Hence in increasing order the
first $p-1$ entries are $T^{\mathrm r}$ (in order), the $p$-th entry is $i$, and the last $s-p$
entries are $B^{\mathrm r}$ (in order). So $r_p=i$. For $a\le p-1$, $r_a=a\in T^{\mathrm r}$ and
$\rho(r_a)=r_a=a$. For $a\ge p+1$, $r_a$ is the $(a-p)$-th smallest element of $B^{\mathrm r}$,
namely $r_a=p+H+(a-p-1)$, whose role is $\rho(r_a)=\bigl(p+H+a-p-1\bigr)-H+1=a$.

The column proof is identical: every element of $L^{\mathrm c}$ is $<q\le j$ and every element of
$R^{\mathrm c}$ is $>q+W-1\ge j$, so the first $q-1$ entries of $C(j)$ are $L^{\mathrm c}$, the
$q$-th is $j$, and the last $t-q$ are $R^{\mathrm c}$; thus $c_q=j$, while $\sigma(c_b)=c_b=b$ for
$b\le q-1$ and $\sigma(c_b)=(q+W+(b-q-1))-W+1=b$ for $b\ge q+1$.
\end{proof}

\begin{lemma}\label{lem:canon}
Let $i\in\Phi^{\mathrm r}$, $j\in\Phi^{\mathrm c}$, and let $R(i)=(r_1<\dots<r_s)$,
$C(j)=(c_1<\dots<c_t)$. Then for all $(a,b)\in[s]\times[t]$,
\[
A[r_a,c_b]=
\begin{cases}
B[i,j], & (a,b)=(p,q),\\
Q[a,b], & (a,b)\ne(p,q).
\end{cases}
\]
In particular $(R(i),C(j))$ is a copy of $Q$ in $A$ if and only if $B[i,j]=1$.
\end{lemma}

\begin{proof}
Fix $(a,b)$. By Lemma~\ref{lem:order}, $r_a=i$ iff $a=p$ (and otherwise $r_a$ is the scaffold row
of role $a$), and $c_b=j$ iff $b=q$ (otherwise $c_b$ is the scaffold column of role $b$). We read
off $A[r_a,c_b]$ from \eqref{eq:def} according to the four position types.

\emph{Case $(a,b)=(p,q)$:} $r_a=i\in\Phi^{\mathrm r}$, $c_b=j\in\Phi^{\mathrm c}$, a free cell, so
$A[r_a,c_b]=B[i,j]$.

\emph{Case $a=p$, $b\ne q$:} $r_a=i$ is a free row and $c_b$ is the scaffold column of role $b$;
by \eqref{eq:def} (free row, scaffold col), $A[r_a,c_b]=Q[p,\sigma(c_b)]=Q[p,b]=Q[a,b]$.

\emph{Case $a\ne p$, $b=q$:} $r_a$ is the scaffold row of role $a$ and $c_b=j$ is a free column;
by \eqref{eq:def} (scaffold row, free col), $A[r_a,c_b]=Q[\rho(r_a),q]=Q[a,q]=Q[a,b]$.

\emph{Case $a\ne p$, $b\ne q$:} both scaffold; by \eqref{eq:def},
$A[r_a,c_b]=Q[\rho(r_a),\sigma(c_b)]=Q[a,b]$.

Thus every entry equals $Q[a,b]$ except the $(p,q)$ entry, which equals $B[i,j]$. The copy
condition $A[r_a,c_b]=Q[a,b]$ for all $(a,b)$ therefore holds iff the sole exceptional entry
agrees, i.e.\ iff $B[i,j]=Q[p,q]=1$.
\end{proof}

We use the full scaffold $A(B)$ later to count by weight, but it can contain unwitnessed arms
when $B$ has a zero row or column. For the unrestricted count we delete
those arms. Let
\[
\mathcal R(B):=\{i\in\Phi^{\mathrm r}:B[i,j]=1\text{ for some }j\in\Phi^{\mathrm c}\},
\qquad
\mathcal C(B):=\{j\in\Phi^{\mathrm c}:B[i,j]=1\text{ for some }i\in\Phi^{\mathrm r}\}.
\]
Thus $\mathcal R(B)$ and $\mathcal C(B)$ are the sets of nonzero rows and columns of $B$.
Define the \emph{pruned scaffold} $\widehat A(B)$ by
\begin{equation}\label{eq:pruned}
\widehat A(B)[r,c]=
\begin{cases}
B[r,c], & r\in\Phi^{\mathrm r},\ c\in\Phi^{\mathrm c},\\[2pt]
Q[p,\sigma(c)], & r\in\mathcal R(B),\ c\notin\Phi^{\mathrm c},\\[2pt]
Q[\rho(r),q], & r\notin\Phi^{\mathrm r},\ c\in\mathcal C(B),\\[2pt]
Q[\rho(r),\sigma(c)], &
 r\notin\Phi^{\mathrm r},\ c\notin\Phi^{\mathrm c},\ B\ne0,\\[2pt]
0, & \text{otherwise}.
\end{cases}
\end{equation}
Equivalently, starting from $A(B)$, we delete the scaffold-column arms in every zero row of
$B$, the scaffold-row arms in every zero column of $B$, and, when $B=0$, the entire scaffold
corner. In particular $\widehat A(0)=0$.

\begin{lemma}\label{lem:pruned}
For every $B\in\{0,1\}^{\Phi^{\mathrm r}\times\Phi^{\mathrm c}}$, the matrix
$\widehat A(B)$ is strongly $Q$-forcing.
\end{lemma}

\begin{proof}
Suppose first that $B[i,j]=1$. Then $i\in\mathcal R(B)$, $j\in\mathcal C(B)$, and $B\ne0$.
Consequently none of the entries used by the canonical submatrix on $R(i)$ and $C(j)$ is deleted:
the free-row/scaffold-column entries are retained because $i\in\mathcal R(B)$, the
scaffold-row/free-column entries because $j\in\mathcal C(B)$, and the scaffold-corner entries
because $B\ne0$. Lemmas~\ref{lem:order} and \ref{lem:canon}, or the same four-case verification
using \eqref{eq:pruned}, therefore show that this submatrix of $\widehat A(B)$ is exactly $Q$.
We call it the canonical copy associated with $(i,j)$.

Every $1$-entry of $\widehat A(B)$ is now witnessed. A free $1$ at $(i,j)$ is witnessed at
position $(p,q)$ of its canonical copy. If $(r,c)$ is a scaffold-row/free-column $1$, then
$c\in\mathcal C(B)$; choose $i$ with $B[i,c]=1$. The canonical copy associated with $(i,c)$
contains $(r,c)$ at position $(\rho(r),q)$, where $Q[\rho(r),q]=1$. Similarly, if $(r,c)$ is a
free-row/scaffold-column $1$, choose $j$ with $B[r,j]=1$; the canonical copy associated with
$(r,j)$ witnesses it at position $(p,\sigma(c))$. Finally, a scaffold-corner $1$ can occur only
when $B\ne0$. Choose any $(i,j)$ with $B[i,j]=1$; its canonical copy contains $(r,c)$ at position
$(\rho(r),\sigma(c))$, a $1$-position of $Q$. These four types exhaust the $1$-entries of
$\widehat A(B)$. (When $B=0$, there are no such entries.) Hence $\widehat A(B)$ is strongly
$Q$-forcing.
\end{proof}

\begin{proof}[Proof of Theorem~\ref{thm:main}]
Fix any $1$-entry $(p,q)$ of $Q$. By Lemma~\ref{lem:pruned}, every one of the $2^{HW}$ free
blocks $B\in\{0,1\}^{\Phi^{\mathrm r}\times\Phi^{\mathrm c}}$ produces a strongly
$Q$-forcing matrix $\widehat A(B)$. The map $B\mapsto\widehat A(B)$ is injective, since $B$ is
recovered by restricting $\widehat A(B)$ to the free block. Thus
\[
\Fstar(m,n,Q)\ge 2^{HW}.
\]
Taking logarithms and using
\[
mn-HW=(s-1)n+(t-1)m-(s-1)(t-1)
\]
gives the deficiency bound. On setting $m=n$ and holding $Q$ fixed,
\[
\frac{(n-s+1)(n-t+1)}{n^2}
\le \frac{\log_2\Fstar(n,n,Q)}{n^2}\le1.
\]
Both outer expressions tend to $1$, so the squeeze theorem gives $\lambda_Q=1$.
\end{proof}

\subsection{The full scaffold}

For the weight-refined results in Section~\ref{sec:weight}, we retain the full matrix $A(B)$ of
\eqref{eq:def}. The following observations specify when its entries are witnessed.

By \eqref{eq:def} the $1$-entries of $A(B)$ are of exactly four kinds:
\emph{free} $1$'s ($r,c$ free with $B[r,c]=1$);
\emph{scaffold-row arms} ($r$ scaffold, $c$ free, $Q[\rho(r),q]=1$);
\emph{scaffold-column arms} ($r$ free, $c$ scaffold, $Q[p,\sigma(c)]=1$);
and \emph{corner} $1$'s ($r,c$ scaffold, $Q[\rho(r),\sigma(c)]=1$).

\begin{lemma}\label{lem:witness}
Let $B\in\{0,1\}^{\Phi^{\mathrm r}\times\Phi^{\mathrm c}}$.
\begin{enumerate}
\item[\textup{(a)}] Every free $1$-entry of $A(B)$ is witnessed.
\item[\textup{(b)}] If every column of $B$ is nonzero, every scaffold-row arm $1$-entry is
witnessed.
\item[\textup{(c)}] If every row of $B$ is nonzero, every scaffold-column arm $1$-entry is
witnessed.
\item[\textup{(d)}] If $B\ne 0$, every corner $1$-entry is witnessed.
\end{enumerate}
\end{lemma}

\begin{proof}
(a) Let $(i,j)$ be a free $1$-entry, so $i\in\Phi^{\mathrm r}$, $j\in\Phi^{\mathrm c}$,
$B[i,j]=1$. By Lemma~\ref{lem:canon} the pair $(R(i),C(j))$ is a copy of $Q$, and in it the
position $(a,b)=(p,q)$ satisfies $r_a=i$, $c_b=j$, $Q[p,q]=1$. Hence $(i,j)$ is witnessed.

(b) Let $(r,c)$ be a scaffold-row arm $1$-entry: $r$ a scaffold row of role $a:=\rho(r)\ne p$,
$c\in\Phi^{\mathrm c}$, and $A[r,c]=Q[a,q]=1$. By hypothesis the column $c$ of $B$ is nonzero, so
choose $i\in\Phi^{\mathrm r}$ with $B[i,c]=1$. Consider $(R(i),C(c))$. Since $B[i,c]=1$,
Lemma~\ref{lem:canon} shows it is a copy of $Q$. Its rows $R(i)=T^{\mathrm r}\cup\{i\}\cup
B^{\mathrm r}$ contain $r$ at role $a$ by Lemma~\ref{lem:order}; its columns $C(c)$ place $c$ at
role $q$. Hence $A[r,c]$ occupies position $(a,q)$ of this copy, where $Q[a,q]=1$. So $(r,c)$ is
witnessed.

(c) Symmetric to (b): let $(r,c)$ be a scaffold-column arm, $r\in\Phi^{\mathrm r}$, $c$ a
scaffold column of role $b:=\sigma(c)\ne q$, $A[r,c]=Q[p,b]=1$. The row $r$ of $B$ is nonzero;
pick $j\in\Phi^{\mathrm c}$ with $B[r,j]=1$. Then $(R(r),C(j))$ is a copy of $Q$ (as
$B[r,j]=1$), it contains $r$ at role $p$ and $c$ at role $b$, and $Q[p,b]=1$, so $(r,c)$ is
witnessed.

(d) Let $(r,c)$ be a corner $1$-entry: $r$ a scaffold row of role $a:=\rho(r)\ne p$, $c$ a
scaffold column of role $b:=\sigma(c)\ne q$, $A[r,c]=Q[a,b]=1$. As $B\ne0$ choose
$i\in\Phi^{\mathrm r}$, $j\in\Phi^{\mathrm c}$ with $B[i,j]=1$. Then $(R(i),C(j))$ is a copy of
$Q$ containing $r$ at role $a$ and $c$ at role $b$, with $Q[a,b]=1$; so $(r,c)$ is witnessed.
\end{proof}

\begin{corollary}\label{cor:good}
If $B$ has no zero row and no zero column, then $A(B)$ is strongly $Q$-forcing.
\end{corollary}

\begin{proof}
By Lemma~\ref{lem:witness}: part (a) witnesses every free $1$; ``no zero column'' gives the
hypothesis of (b), witnessing every scaffold-row arm $1$; ``no zero row'' gives the hypothesis of
(c), witnessing every scaffold-column arm $1$; and (since $H,W\ge1$) ``no zero row'' forces
$B\ne0$, so (d) witnesses every corner $1$. As every $1$-entry of $A(B)$ is a free $1$, a
scaffold-row arm, a scaffold-column arm, or a corner $1$, all are witnessed.
\end{proof}

\section{Full scaffolds and the maximum number of ones}\label{sec:special}

The pruned scaffold proves the unrestricted lower bound, while the full scaffold has a weight
formula used below. When the anchor is the unique $1$ in its row or column, some
of the nonzero-row or nonzero-column conditions in Corollary~\ref{cor:good} can be dropped. The
next lemma records the variants needed in Section~\ref{sec:weight}.

\begin{lemma}\label{lem:special}
Let $(p,q)$ be a $1$-entry of $Q$ and form $A(B)$ with anchor $(p,q)$ as in Section~\ref{sec:construction}.
\begin{enumerate}
\item[\textup{(a)}] If $Q[p,q]$ is the only $1$-entry in row $p$ of $Q$, then $A(B)$ is strongly
$Q$-forcing whenever every column of $B$ is nonzero. No condition on the rows of $B$ is needed.
\item[\textup{(b)}] If $Q[p,q]$ is the only $1$-entry in column $q$ of $Q$, then $A(B)$ is
strongly $Q$-forcing whenever every row of $B$ is nonzero.
\item[\textup{(c)}] If $Q[p,q]$ is the only $1$-entry in both its row and its column, then $A(B)$
is strongly $Q$-forcing for every nonzero $B$.
\end{enumerate}
\end{lemma}

\begin{proof}
Recall that every $1$-entry of $A(B)$ is a free $1$, a scaffold-row arm ($Q[\rho(r),q]=1$), a
scaffold-column arm ($Q[p,\sigma(c)]=1$), or a corner $1$.

(a) If $Q[p,q]$ is the only $1$-entry in row $p$, then $Q[p,\sigma(c)]=0$ for every scaffold
column $c$ (as $\sigma(c)\ne q$), so $A(B)$ has no scaffold-column arm $1$-entries. The remaining
kinds are handled by Lemma~\ref{lem:witness}: part (a) witnesses every free $1$ unconditionally;
if every column of $B$ is nonzero, part (b) witnesses every scaffold-row arm; and a nonzero column
forces $B\ne0$, so part (d) witnesses every corner $1$. Hence $A(B)$ is strongly $Q$-forcing, with
no condition imposed on the rows of $B$.

(b) If $Q[p,q]$ is the only $1$-entry in column $q$, then $A(B)$ has no scaffold-row arm
$1$-entries. Lemma~\ref{lem:witness}(a) witnesses every free $1$-entry, part (c) witnesses
every scaffold-column arm because each row of $B$ is nonzero, and part (d) witnesses every
corner $1$-entry because a nonzero row implies $B\ne0$. Thus $A(B)$ is strongly
$Q$-forcing.

(c) If $Q[p,q]$ is the only $1$-entry in both its row and its column, then $A(B)$ has neither
scaffold-row nor scaffold-column arm $1$-entries; only free and corner $1$'s remain. Part (a) of
Lemma~\ref{lem:witness} witnesses the free $1$'s, and if $B\ne0$ part (d) witnesses the corner
$1$'s.
\end{proof}

We turn to the maximum number of $1$-entries a strongly $Q$-forcing matrix can have when $Q$ has
a full row and a full column.

\begin{lemma}\label{lem:maxweight}
Let $Q$ be an $s\times t$ pattern such that row $i$ and column $j$ of $Q$ are all-one. Then, for
$m\ge s$ and $n\ge t$,
\[
M(m,n,Q)=mn-z(Q),
\]
where $z(Q)$ is the number of zero entries of $Q$.
\end{lemma}

\begin{proof}
\emph{Upper bound.} Let $A$ be an $m\times n$ strongly $Q$-forcing matrix. If $A$ is the all-zero
matrix then $|A|_1=0\le mn-z(Q)$. Otherwise $A$ has a $1$-entry, which is witnessed and hence lies
in a copy of $Q$ in $A$. That copy reproduces all $z(Q)$ zero entries of $Q$, so $A$ has at least
$z(Q)$ zero entries and $|A|_1\le mn-z(Q)$. Thus $M(m,n,Q)\le mn-z(Q)$.

\emph{Lower bound.} Choose the anchor $(p,q)=(i,j)$ in the construction and let $B$ be the
all-one block in $\{0,1\}^{\Phi^{\mathrm r}\times\Phi^{\mathrm c}}$ (so $B$ has no zero row and no
zero column). Inspecting \eqref{eq:def}: every free cell equals $1$; every scaffold-row arm
equals $Q[\rho(r),j]=1$ since column $j$ of $Q$ is all-one; every scaffold-column arm equals
$Q[i,\sigma(c)]=1$ since row $i$ of $Q$ is all-one. The only possible zeros lie in the scaffold
cells, whose values $Q[\rho(r),\sigma(c)]$ range, as $(\rho(r),\sigma(c))$ runs over
$([s]\setminus\{i\})\times([t]\setminus\{j\})$, over exactly the entries of $Q$ outside row $i$
and column $j$. Since every zero of $Q$ lies off the all-one row $i$ and column $j$, the scaffold
cells reproduce each of the $z(Q)$ zeros of $Q$ exactly once. Hence $A(B)$ has exactly $z(Q)$ zero
entries, i.e.\ $|A(B)|_1=mn-z(Q)$. By Corollary~\ref{cor:good}, $A(B)$ is strongly $Q$-forcing, so
$M(m,n,Q)\ge mn-z(Q)$.

Combining the two bounds gives $M(m,n,Q)=mn-z(Q)$.
\end{proof}

\section{Exact families and sharpness}\label{sec:tightness}

The universal bound in Theorem~\ref{thm:main} is attained by every pattern having a single
$1$-entry. Thus, among all nonzero patterns of prescribed dimensions, the minimum possible
strong-forcing count is determined exactly. We prove that extremal statement below, together with
closed formulas for several other column and rectangular families.

Throughout, $Q$ is an $s\times t$ pattern with at least one $1$-entry, and we write
$H=m-s+1$ and $W=n-t+1$. We denote by
$J_{s,1}$ the $s\times1$ all-ones pattern (a single column of $s$ ones), by $J_{1,t}$ the $1\times t$
all-ones pattern, and by $J_{s,t}$ the $s\times t$ all-ones pattern.

\subsection{Transpose symmetry}

We first record a symmetry used repeatedly below. For an $m\times n$ matrix $A$ let $A^{\mathsf T}$
denote its $n\times m$ transpose, $A^{\mathsf T}[j,i]=A[i,j]$, and for an $s\times t$ pattern $Q$ let
$Q^{\mathsf T}$ denote its $t\times s$ transpose.

\begin{lemma}\label{lem:transpose}
For every $s\times t$ pattern $Q$ with a $1$-entry and all $m\ge s$, $n\ge t$,
\[
\Fstar(m,n,Q)=\Fstar(n,m,Q^{\mathsf T}).
\]
\end{lemma}

\begin{proof}
If rows $r_1<\dots<r_s$ and columns $c_1<\dots<c_t$ form a copy of $Q$ in $A$, then
\[
A^{\mathsf T}[c_b,r_a]=A[r_a,c_b]=Q[a,b]=Q^{\mathsf T}[b,a].
\]
Thus the same columns and rows form a copy of $Q^{\mathsf T}$ in $A^{\mathsf T}$. This
correspondence maps $(i,j)$ to $(j,i)$ and preserves whether an entry is witnessed. Hence
$A$ is strongly $Q$-forcing if and only if $A^{\mathsf T}$ is strongly
$Q^{\mathsf T}$-forcing. Transposition is a bijection, so the counts are equal.
\end{proof}

\subsection{The all-ones row and column families}

\begin{lemma}\label{lem:colchar}
Let $Q=J_{s,1}$ and let $A$ be an $m\times n$ matrix with $m\ge s$. Then $A$ is strongly
$Q$-forcing if and only if every column of $A$ contains either $0$ ones or at least $s$ ones.
\end{lemma}

\begin{proof}
Since $J_{s,1}$ is a single column of $s$ ones, a copy of $Q$ in $A$ is a choice of $s$ rows
$r_1<\dots<r_s$ and a single column $c$ with $A[r_a,c]=1$ for all $a$; this copy witnesses
the $1$-entries $(r_1,c),\dots,(r_s,c)$, all of which lie in column $c$.

($\Leftarrow$) Suppose every column has $0$ or at least $s$ ones, and let $(i,j)$ be a $1$-entry.
Column $j$ then contains a $1$, so it has at least $s$ ones, say in rows including $i$. Choosing $i$
together with $s-1$ further one-rows of column $j$ and ordering them increasingly yields a copy of
$Q$ in column $j$ that witnesses $(i,j)$. Hence $A$ is strongly $Q$-forcing.

($\Rightarrow$) Suppose some column $j$ has $k$ ones with $1\le k\le s-1$, and let $(i,j)$ be one of
them. Any $J_{s,1}$-copy occupies a single column; to witness $(i,j)$ that column must be $j$, and
the copy needs $s$ rows on which column $j$ is all $1$. Only $k<s$ such rows exist, so no copy
witnesses $(i,j)$ and $A$ is not strongly $Q$-forcing. Contrapositively, if $A$ is strongly
$Q$-forcing then every column has $0$ or at least $s$ ones.
\end{proof}

\begin{theorem}\label{thm:colformula}
For all $m\ge s\ge1$ and $n\ge1$,
\[
\Fstar(m,n,J_{s,1})=\Bigl(2^{m}-\sum_{i=1}^{s-1}\binom mi\Bigr)^{\!n}
=\Bigl(1+\sum_{k=s}^{m}\binom mk\Bigr)^{\!n},
\]
and, by transpose symmetry, for all $n\ge t\ge1$ and $m\ge1$,
\[
\Fstar(m,n,J_{1,t})=\Bigl(2^{n}-\sum_{i=1}^{t-1}\binom ni\Bigr)^{\!m}
=\Bigl(1+\sum_{k=t}^{n}\binom nk\Bigr)^{\!m}.
\]
\end{theorem}

\begin{proof}
By Lemma~\ref{lem:colchar}, $A$ is strongly $J_{s,1}$-forcing if and only if each of its $n$
columns, as a vector in $\{0,1\}^m$, has weight $0$ or weight at least $s$. This is a condition on
each column separately, with no interaction between columns (a $J_{s,1}$-copy lives in one column),
so the columns may be chosen independently and
\[
\Fstar(m,n,J_{s,1})=a(m,s)^{n},\qquad
a(m,s):=\#\{v\in\{0,1\}^m: \mathrm{wt}(v)=0\text{ or }\mathrm{wt}(v)\ge s\}.
\]
Counting admissible columns by weight, $a(m,s)=\binom m0+\sum_{k=s}^{m}\binom mk=1+\sum_{k=s}^m\binom mk$.
Since $\sum_{k=0}^m\binom mk=2^m$, removing the weights $1,\dots,s-1$ gives the equivalent form
$a(m,s)=2^m-\sum_{i=1}^{s-1}\binom mi$. This proves the column formula; the row formula follows from
Lemma~\ref{lem:transpose} with $Q=J_{1,t}$, since $J_{1,t}^{\mathsf T}=J_{t,1}$ and the column
formula applies with $m,n,s$ replaced by $n,m,t$.
\end{proof}

\begin{theorem}\label{thm:onezero}
Let $Q$ be an $s\times1$ pattern with exactly one zero, in position $p$, where $s\ge2$.
Then the following formulas hold for all $m\ge s$ and $n\ge1$.

If $p=1$ or $p=s$, then
\[
\Fstar(m,n,Q)
=
\left(
2^{m-1}
-
\sum_{i=1}^{s-2}\binom{m-1}{i}
\right)^n.
\]

If $1<p<s$, then
\[
\Fstar(m,n,Q)
=
\left(
2^m
-
\sum_{i=1}^{s-1}\binom{m}{i}
\right)^n.
\]
\end{theorem}

\begin{proof}
Since $Q$ is a column pattern, the columns of the ambient $m\times n$
matrix may be chosen independently. Thus it suffices to count the admissible
columns of length $m$.

First suppose $p=1$. A nonzero column is
strongly $Q$-forcing if and only if its first entry is $0$ and it contains
at least $s-1$ ones. Indeed, the topmost $1$ can only occupy the first
$1$-position of $Q$, so it needs a zero above it. Conversely, if the first entry is zero and
the column has at least $s-1$ ones, then any prescribed $1$ can be selected with $s-2$ other
$1$'s and the first entry to form a copy of $Q$. Hence the number of admissible columns is
\[
1+\sum_{k=s-1}^{m-1}\binom{m-1}{k}
=
2^{m-1}-\sum_{k=1}^{s-2}\binom{m-1}{k}.
\]
This gives the first formula.
If $p=s$, the same proof uses the bottommost $1$ and the last entry of the column in place of
the topmost $1$ and the first entry. The number of admissible columns is unchanged.

Now suppose $1<p<s$. Let $v\in\{0,1\}^m$ be a column with exactly $k$
ones, and write it in gap form
\[
g_0\,1\,g_1\,1\,g_2\,1\cdots 1\,g_{k-1}\,1\,g_k,
\]
where $g_i$ is the number of zeros after exactly $i$ ones. A zero in gap
$g_i$ has $i$ ones above it and $k-i$ ones below it. Hence it can serve
as the zero entry of a copy of $Q$ if and only if
\[
p-1\le i\le k-s+p.
\]
Moreover, the existence of one such zero is sufficient to witness \emph{every} $1$ in the
column. Indeed, a prescribed $1$ above the zero can be included among $p-1$ selected ones above
it, together with any $s-p$ ones below; a prescribed $1$ below the zero is handled symmetrically.
The two inequalities above guarantee enough ones on both sides in either case.
Conversely, a copy witnessing any $1$ contains a zero in one of these gaps. We call a gap $g_i$ \emph{good} if $p-1\le i\le k-s+p$. Thus a nonzero
column is admissible if and only if at least one good gap in the column contains a zero.

If $1\le k\le s-2$, then no such zero can exist, so all columns with $k$
ones are inadmissible. These contribute
\[
\sum_{k=1}^{s-2}\binom{m}{k}
\]
nonzero inadmissible columns.

Now assume $k\ge s-1$. The inadmissible columns are those for which every
zero lies outside the good gaps
\[
g_{p-1},g_p,\ldots,g_{k-s+p}.
\]
Thus all zeros must lie in the remaining gaps
\[
g_0,\ldots,g_{p-2}
\qquad\text{and}\qquad
g_{k-s+p+1},\ldots,g_k.
\]
There are
\[
(p-1)+(s-p)=s-1
\]
such gaps. Therefore, for fixed $k\ge s-1$, the number of inadmissible
columns with exactly $k$ ones is the number of weak compositions of
$m-k$ into $s-1$ parts, namely
\[
\binom{m-k+s-2}{s-2}.
\]
Hence the total number of nonzero inadmissible columns is
\[
\sum_{k=1}^{s-2}\binom{m}{k}
+
\sum_{k=s-1}^{m}\binom{m-k+s-2}{s-2}.
\]
By the hockey-stick identity,
\[
\sum_{k=s-1}^{m}\binom{m-k+s-2}{s-2}
=
\binom{m}{s-1}.
\]
Thus the number of nonzero inadmissible columns is
\[
\sum_{k=1}^{s-1}\binom{m}{k}.
\]
Since the zero column is admissible, the number of admissible columns is
\[
2^m-\sum_{k=1}^{s-1}\binom{m}{k}.
\]
This gives the second formula.
\end{proof}

\begin{corollary}\label{cor:boundaryzeros}
Let $Q$ be a nonzero $s\times1$ pattern whose zeros form two possibly empty
blocks at the top and bottom, with $k$ zeros in total. Then, for $m\ge s$ and $n\ge1$,
\[
\Fstar(m,n,Q)=\Fstar(m-k,n,J_{s-k,1}).
\]
Consequently,
\[
d_Q(m,n)=d_{J_{s-k,1}}(m-k,n)+kn.
\]
\end{corollary}

\begin{proof}
Write $Q=0^a1^{s-k}0^b$, where $a,b\ge0$ and $a+b=k$. A nonzero column
$v\in\{0,1\}^m$ is strongly $Q$-forcing if and only if its first $a$ and last $b$ entries are
zero and its remaining $m-k$ entries contain at least $s-k$ ones. For necessity, apply a
witnessing copy to the topmost and bottommost $1$ of $v$. The topmost $1$ must occupy the first
$1$-position of $Q$ and therefore needs $a$ zeros above it. Similarly, the bottommost $1$
needs $b$ zeros below it, and any copy supplies $s-k$ ones. Conversely, under
the stated conditions, any prescribed $1$ can be included among $s-k$ selected ones, together
with the first $a$ and last $b$ entries as the required zeros. The zero column is admissible as
well.

Deleting the forced top $a$ and bottom $b$ entries is therefore a bijection between admissible
length-$m$ columns for $Q$ and admissible length-$(m-k)$ columns for $J_{s-k,1}$. Since columns
are independent, the first identity follows. Finally,
\[
d_Q(m,n)=mn-\log_2\Fstar(m,n,Q)
=kn+d_{J_{s-k,1}}(m-k,n).
\]
\end{proof}

\subsection{Singleton patterns: exact extremizers}

\begin{definition}\label{def:corner}
For $(p,q)\in[s]\times[t]$, let $E_{s,t}^{p,q}$ be the $s\times t$ \emph{singleton pattern}
whose sole $1$ is in position $(p,q)$. Write $E_{s,t}:=E_{s,t}^{1,1}$ for the top-left
singleton pattern.
\end{definition}

\begin{lemma}\label{lem:cornerchar}
Let $m\ge s$ and $n\ge t$, put $H=m-s+1$ and $W=n-t+1$, and let $A$ be an $m\times n$
$(0,1)$-matrix. Then $A$ is strongly $E_{s,t}^{p,q}$-forcing if and only if its support is
contained in
\[
I_p\times J_q,\qquad
I_p:=\{p,\dots,p+H-1\},\qquad J_q:=\{q,\dots,q+W-1\}.
\]
\end{lemma}

\begin{proof}
If a $1$-entry $(i,j)$ is witnessed, it occupies position $(p,q)$ in a copy. The copy needs
$p-1$ rows before $i$ and $s-p$ rows after it, so
$p\le i\le m-s+p=p+H-1$. Similarly, $q\le j\le q+W-1$. This proves necessity.

Conversely, suppose the support of $A$ lies in $I_p\times J_q$, and let $(i,j)$ be a
$1$-entry. Choose the rows
\[
1,\dots,p-1,\ i,\ p+H,\dots,m
\]
and the columns
\[
1,\dots,q-1,\ j,\ q+W,\dots,n.
\]
They are strictly increasing and place $(i,j)$ in role $(p,q)$. Every other selected cell has
either a row outside $I_p$ or a column outside $J_q$, and hence is zero. The selected submatrix
is therefore exactly $E_{s,t}^{p,q}$ and witnesses $(i,j)$. The all-zero matrix is covered
vacuously.
\end{proof}

\begin{theorem}\label{thm:corner}
Let $(p,q)\in[s]\times[t]$, $m\ge s$, and $n\ge t$. Put $H=m-s+1$ and
$W=n-t+1$. Then
\[
\Fstar(m,n,E_{s,t}^{p,q})=2^{HW}.
\]
Consequently,
\[
\min_{\substack{Q\in\{0,1\}^{s\times t}\\Q\ne0}}\Fstar(m,n,Q)=2^{HW},
\qquad
\max_{\substack{Q\in\{0,1\}^{s\times t}\\Q\ne0}}d_Q(m,n)=mn-HW.
\]
\end{theorem}

\begin{proof}
Lemma~\ref{lem:cornerchar} gives a bijection between strongly
$E_{s,t}^{p,q}$-forcing matrices and arbitrary $(0,1)$-fillings of the $H\times W$ rectangle
$I_p\times J_q$, proving the first formula. Theorem~\ref{thm:main} gives the same quantity as a
lower bound for every nonzero $s\times t$ pattern, so the first extremal identity follows.
Applying $d_Q(m,n)=mn-\log_2\Fstar(m,n,Q)$ gives the second.
\end{proof}

\subsection{Full-height all-ones patterns}

\begin{theorem}\label{thm:fullheight}
For all $m\ge1$ and $n\ge t\ge1$,
\[
\Fstar(m,n,J_{m,t})
=1+\sum_{k=t}^{n}\binom nk
=2^n-\sum_{k=1}^{t-1}\binom nk,
\]
and in particular $\Fstar(m,n,J_{m,1})=2^n$ and
$\Fstar(m,n,J_{m,2})=2^n-n$.
\end{theorem}

\begin{proof}
A copy of $J_{m,t}$ uses all $m$ rows and $t$ all-one columns. Hence a $1$-entry is witnessed
if and only if its column is all-one and belongs to a set of at least $t$ all-one columns.
It follows that a strongly $J_{m,t}$-forcing matrix is either the all-zero matrix, or consists
of $k\ge t$ all-one columns with every other column zero. Choosing the $k$ columns and summing
over $k$ gives $1+\sum_{k=t}^n\binom nk$. The second form follows from the binomial theorem.
\end{proof}

\section{Counting by the number of ones}\label{sec:weight}

We now refine the count $\Fstar(m,n,Q)$ according to the number of $1$-entries. For an integer
$w\ge0$ let $\Fstar(m,n,Q,w)$ denote the number of $m\times n$ strongly $Q$-forcing matrices $A$
with $|A|_1=w$, so that $\Fstar(m,n,Q)=\sum_{w\ge0}\Fstar(m,n,Q,w)$. Write $o(Q):=st-z(Q)$ for the
number of $1$-entries of $Q$, and for an anchor $(p,q)$ set
\[
\mu_p:=\sum_{b=1}^{t}Q[p,b],\qquad \nu_q:=\sum_{a=1}^{s}Q[a,q]
\]
(the number of $1$'s in row $p$ and in column $q$ of $Q$). Finally let
\[
h(\alpha):=-\alpha\log_2\alpha-(1-\alpha)\log_2(1-\alpha)\quad(0<\alpha<1),\qquad h(0):=h(1):=0
\]
denote the binary entropy function. Theorem~\ref{thm:minweight} gives the exact count at
$w=o(Q)$, and Theorem~\ref{thm:entropy} gives the exponential growth rate when $w/(mn)$ tends
to a constant in $(0,1)$.

\begin{lemma}\label{lem:weightAB}
Let $m\ge s$ and $n\ge t$. Fix an anchor $(p,q)$ and form $A(B)$ as in
Section~\ref{sec:construction}. Put $H=m-s+1$ and $W=n-t+1$. Then for every free block $B$,
\[
|A(B)|_1=|B|_1+c_0,\qquad
c_0:=H(\mu_p-1)+W(\nu_q-1)+\bigl(o(Q)-\mu_p-\nu_q+1\bigr).
\]
Each summand is nonnegative, so $c_0\ge0$; if $(p,q)$ is the unique $1$-entry in its row and in its
column then $c_0=o(Q)-1$; and for fixed $Q$, $c_0\le(t-1)H+(s-1)W+o(Q)=O(m+n)$.
\end{lemma}

\begin{proof}
By \eqref{eq:def} the cells of $A(B)$ split into the four types listed there.

The free cells $\Phi^{\mathrm r}\times\Phi^{\mathrm c}$ carry the bits of $B$, contributing
$|B|_1$ ones.

For a free row $r\in\Phi^{\mathrm r}$ and a scaffold column $c\notin\Phi^{\mathrm c}$,
$A[r,c]=Q[p,\sigma(c)]$ depends only on $\sigma(c)$. As $c$ ranges over the $t-1$ scaffold columns,
$\sigma(c)$ ranges bijectively over $[t]\setminus\{q\}$ (Section~\ref{sec:construction}), so each fixed free row carries
exactly $|\{b\ne q:Q[p,b]=1\}|=\mu_p-1$ ones in scaffold columns. Over the $H$ free rows this is
$H(\mu_p-1)$ ones. Symmetrically, the scaffold-row/free-column cells contribute $W(\nu_q-1)$.

For scaffold cells ($r\notin\Phi^{\mathrm r}$, $c\notin\Phi^{\mathrm c}$),
$A[r,c]=Q[\rho(r),\sigma(c)]$, and $(\rho(r),\sigma(c))$ runs bijectively over
$([s]\setminus\{p\})\times([t]\setminus\{q\})$ (Section~\ref{sec:construction}). The number of $1$'s of $Q$ in these
positions is, by inclusion--exclusion, $o(Q)-\mu_p-\nu_q+1$: from the $o(Q)$ ones of $Q$ remove
those in row $p$ ($\mu_p$ of them) and those in column $q$ ($\nu_q$ of them), then restore the entry
$(p,q)$, removed twice, which is a $1$. So the scaffold cells contribute $o(Q)-\mu_p-\nu_q+1$ ones.

Summing the four contributions gives $|A(B)|_1=|B|_1+c_0$. Each summand of $c_0$ is nonnegative:
$\mu_p,\nu_q\ge1$ because $Q[p,q]=1$, and $o(Q)-\mu_p-\nu_q+1$ counts the $1$'s of $Q$ outside row
$p$ and column $q$. If $(p,q)$ is the unique $1$ in its row and column then $\mu_p=\nu_q=1$, whence
$c_0=o(Q)-1$. Finally $\mu_p\le t$ and $\nu_q\le s$, so $c_0\le(t-1)H+(s-1)W+o(Q)$, which is
$O(m+n)$ for fixed $Q$.
\end{proof}

\begin{lemma}\label{lem:Nk}
Let $H,W$ be positive integers, and let $N_k(H,W)$ be the number of matrices
$B\in\{0,1\}^{H\times W}$ with $|B|_1=k$ that have no zero row and no zero column. Then for
every integer $0\le k\le HW$,
\[
N_k(H,W)\;\ge\;\binom{HW}{k}-H\binom{(H-1)W}{k}-W\binom{H(W-1)}{k},
\]
with the convention $\binom{a}{k}=0$ when $k>a$.
\end{lemma}

\begin{proof}
There are $\binom{HW}{k}$ matrices in $\{0,1\}^{H\times W}$ with exactly $k$ ones. Such a matrix
fails the condition iff some row or some column is entirely zero. If a prescribed row is zero, the
$k$ ones lie among the remaining $(H-1)W$ cells, so there are at most $\binom{(H-1)W}{k}$ such
matrices for each of the $H$ rows; likewise at most $\binom{H(W-1)}{k}$ for each of the $W$ columns.
By the union bound the number of failing matrices is at most
$H\binom{(H-1)W}{k}+W\binom{H(W-1)}{k}$, and subtracting from $\binom{HW}{k}$ gives the claim.
\end{proof}

\begin{lemma}\label{lem:entropy}
For every positive integer $N$ and every integer $0\le k\le N$ one has
$\binom{N}{k}\le 2^{N h(k/N)}$, and for $1\le k\le N-1$,
$\binom{N}{k}\ge \tfrac{1}{N+1}\,2^{N h(k/N)}$. Consequently, if $N\to\infty$ and $k=k(N)$ satisfies
$k/N\to\beta\in(0,1)$, then $N^{-1}\log_2\binom{N}{k}\to h(\beta)$.
\end{lemma}

\begin{proof}
For $k\in\{0,N\}$ the upper bound reads $\binom Nk=1\le2^{N\cdot0}=1$. For $1\le k\le N-1$ put
$\beta:=k/N\in(0,1)$. Since $\beta^k(1-\beta)^{N-k}=2^{k\log_2\beta+(N-k)\log_2(1-\beta)}
=2^{-Nh(\beta)}$, the binomial theorem gives
\[
1=\bigl(\beta+(1-\beta)\bigr)^N=\sum_{j=0}^N\binom Nj\beta^j(1-\beta)^{N-j}
\ \ge\ \binom Nk\beta^k(1-\beta)^{N-k}=\binom Nk\,2^{-Nh(\beta)},
\]
so $\binom Nk\le2^{Nh(\beta)}$. For the lower bound consider the terms
$b_j:=\binom Nj\beta^j(1-\beta)^{N-j}$ ($0\le j\le N$), which are nonnegative and sum to $1$. For
$1\le j\le N$,
\[
\frac{b_j}{b_{j-1}}=\frac{(N-j+1)\beta}{j(1-\beta)}\ \ge\ 1
\iff j\le(N+1)\beta=k+\beta .
\]
As $0<\beta<1$, this holds for $j\le k$ and fails for $j\ge k+1$; hence $b_0\le\dots\le b_k$ and
$b_k\ge b_{k+1}\ge\dots\ge b_N$, so $b_k=\max_j b_j$. Since the $N+1$ terms sum to $1$,
$b_k\ge\frac1{N+1}$, i.e.\ $\binom Nk=b_k\,2^{Nh(\beta)}\ge\frac1{N+1}2^{Nh(\beta)}$. Taking $\log_2$
and dividing by $N$,
\[
h(k/N)+\frac1N\log_2\frac1{N+1}\ \le\ \frac1N\log_2\binom Nk\ \le\ h(k/N).
\]
If $k/N\to\beta\in(0,1)$ then $h(k/N)\to h(\beta)$ by continuity of $h$ on $(0,1)$, and
$\frac1N\log_2\frac1{N+1}\to0$; the squeeze gives $N^{-1}\log_2\binom Nk\to h(\beta)$.
\end{proof}

\begin{theorem}\label{thm:minweight}
Let $Q$ be any nonzero $s\times t$ pattern, and let $r(Q)$ and $c(Q)$ be its numbers of
nonzero rows and nonzero columns. Then for all $m\ge s$, $n\ge t$,
\[
\Fstar(m,n,Q,w)=
\begin{cases}
1, & w=0,\\[2pt]
0, & 1\le w\le o(Q)-1,\\[2pt]
\dbinom{H+r(Q)-1}{r(Q)}\dbinom{W+c(Q)-1}{c(Q)}, & w=o(Q),
\end{cases}
\]
where $H=m-s+1$ and $W=n-t+1$. In particular, if $Q$ has no zero row or column, the last
quantity is $\binom ms\binom nt$.
\end{theorem}

\begin{proof}
\emph{The case $w=0$.} The all-zero matrix is the unique $m\times n$ matrix of weight $0$, and it is
strongly $Q$-forcing vacuously; so $\Fstar(m,n,Q,0)=1$.

\emph{The case $1\le w\le o(Q)-1$.} Let $A$ be strongly $Q$-forcing with $|A|_1=w\ge1$. Choose a
$1$-entry of $A$; being witnessed, it lies in a copy $(r_\bullet,c_\bullet)$ of $Q$. The copy
reproduces $Q$, so the $o(Q)$ positions $(r_a,c_b)$ with $Q[a,b]=1$ are $1$-entries of $A$, and they
are distinct because $(a,b)\mapsto(r_a,c_b)$ is injective (the $r_a$, and the $c_b$, are strictly
increasing). Hence $|A|_1\ge o(Q)>w$, a contradiction; so no such $A$ exists and
$\Fstar(m,n,Q,w)=0$.

\emph{The case $w=o(Q)$.} Let
\[
a_1<\dots<a_r,\qquad b_1<\dots<b_c
\]
be the indices of the nonzero rows and columns of $Q$, where $r=r(Q)$ and $c=c(Q)$. If $A$ is
strongly $Q$-forcing of weight $o(Q)$, choose a witnessed $1$-entry and one copy of $Q$ that
witnesses it. The copy contributes $o(Q)$ distinct $1$-entries, so these are all the
$1$-entries of $A$. Thus $A$ is determined by the ambient images
$x_1<\dots<x_r$ of the nonzero pattern rows and $y_1<\dots<y_c$ of the nonzero pattern
columns. Conversely, any such row and column sets that can be extended to a full embedding of
the $s$ rows and $t$ columns of $Q$ determine a weight-$o(Q)$ matrix; completing the embedding
produces one copy that witnesses all its $1$-entries.

It remains to count the feasible nonzero-row images. The set
$x_1<\dots<x_r$ extends to an increasing image of all $s$ pattern rows if and only if
\[
0\le x_1-a_1\le x_2-a_2\le\dots\le x_r-a_r\le m-s=H-1.
\]
There must be at least $a_1-1$ rows before $x_1$, at least $a_{i+1}-a_i-1$ rows between
$x_i$ and $x_{i+1}$, and at least $s-a_r$ rows after $x_r$. These conditions are equivalent
to the displayed inequalities. Conversely, under these inequalities we can choose the required
number of unused rows in each of these gaps and thereby extend the prescribed images to all
$s$ pattern rows. Hence feasible row sets are in bijection with weakly increasing $r$-tuples from
$\{0,\dots,H-1\}$, of which there are $\binom{H+r-1}{r}$. For the columns, the feasible images
$y_1<\dots<y_c$ are characterized by
\[
0\le y_1-b_1\le\dots\le y_c-b_c\le n-t=W-1.
\]
They are therefore counted by $\binom{W+c-1}{c}$. Finally, the support of $A$ recovers its
nonzero ambient rows and columns,
so different feasible pairs cannot produce the same matrix. Multiplying the two counts proves
the last case.
\end{proof}

\begin{theorem}\label{thm:equality}
Let $Q$ be a nonzero $s\times t$ pattern, $m\ge s$, and $n\ge t$. Put
$H=m-s+1$ and $W=n-t+1$, and let $r(Q)$ and $c(Q)$ be the numbers of nonzero rows and columns
of $Q$. Equality holds in Theorem~\ref{thm:main},
\[
\Fstar(m,n,Q)=2^{HW},
\]
if and only if
\[
\bigl(H=1\ \text{or}\ r(Q)=1\bigr)
\qquad\text{and}\qquad
\bigl(W=1\ \text{or}\ c(Q)=1\bigr).
\]
Thus, when $H,W>1$, the extremizers are exactly the singleton patterns.
\end{theorem}

\begin{proof}
The injection $B\mapsto\widehat A(B)$ in the proof of Theorem~\ref{thm:main} sends each of the
$HW$ blocks having exactly one $1$ to a matrix of weight $o(Q)$: its nonzero entries are exactly
one canonical copy of the $1$-entries of $Q$. No other block has that weight. Indeed,
$\widehat A(0)=0$, while if $B$ has at least two $1$'s, a canonical copy associated with one
of them contributes $o(Q)$ entries and a second free $1$ lies outside that copy. Hence, if
$\Fstar(m,n,Q)=2^{HW}$, the injection is surjective and Theorem~\ref{thm:minweight} gives
\[
\binom{H+r(Q)-1}{r(Q)}\binom{W+c(Q)-1}{c(Q)}=HW. \tag{*}
\]
The first binomial coefficient is at least $H$, since the weakly increasing
$r(Q)$-tuples from $\{0,\dots,H-1\}$ include the $H$ constant tuples. Equality holds exactly
when $H=1$ or $r(Q)=1$; if both exceed $1$, a nonconstant tuple exists. The analogous statement
holds for the second factor and $W$. Since both factors in (*) attain their respective lower
bounds, the two stated conditions are necessary.

For sufficiency, first suppose $H=W=1$. Then the ambient matrix has the same dimensions as
$Q$, so its only possible copy of $Q$ is the whole matrix. The only strongly $Q$-forcing
matrices are therefore $0$ and $Q$, giving $\Fstar=2=2^{HW}$.

Next suppose $H=1<W$ and $c(Q)=1$. Let column $q$ be the unique nonzero column of $Q$, and let
$v\in\{0,1\}^s$ be that column. Since $m=s$, every copy uses all rows. It follows that every
nonzero column of a strongly $Q$-forcing matrix must equal $v$ and must lie in the $W$-column
interval $\{q,\dots,q+W-1\}$ in which it can occupy role $q$. Conversely, any subset of these
$W$ positions may be filled with $v$, with all other columns zero: for each chosen position,
use columns $1,\dots,q-1$ and $q+W,\dots,n$ as the zero scaffold columns of a copy. Thus there
are exactly $2^W=2^{HW}$ matrices. The case $W=1<H$ and $r(Q)=1$ follows by transposition.
Finally, if $H,W>1$, the two conditions give $r(Q)=c(Q)=1$, so $Q$ is a singleton pattern and
Theorem~\ref{thm:corner} applies.
\end{proof}

\begin{theorem}\label{thm:entropy}
Let $Q$ be an $s\times t$ pattern with at least one $1$-entry, and fix $\alpha\in(0,1)$. Suppose
$m,n\to\infty$ with $m/n$ bounded away from $0$ and $\infty$, and let $w=w(m,n)$ be any sequence of
integers with $w/(mn)\to\alpha$. Then
\[
\frac{\log_2\Fstar(m,n,Q,w)}{mn}\;\longrightarrow\;h(\alpha).
\]
\end{theorem}

\begin{proof}
\emph{Upper bound.} Since $\Fstar(m,n,Q,w)\le\binom{mn}{w}$, Lemma~\ref{lem:entropy} gives
\[
\limsup_{m,n\to\infty}\frac{\log_2\Fstar(m,n,Q,w)}{mn}\le h(\alpha).
\]

\emph{Lower bound.} Put $H=m-s+1$ and $W=n-t+1$. Fix a $1$-entry $(p,q)$ of $Q$, let $c_0$
be as in Lemma~\ref{lem:weightAB}, and set $k=w-c_0$. For all sufficiently large $m,n$, the
construction is defined and
\[
c_0=O(m+n)=o(mn),\qquad \frac{HW}{mn}\longrightarrow1,\qquad
\frac{k}{HW}\longrightarrow\alpha.
\]
In particular, $1\le k\le HW-1$. Corollary~\ref{cor:good} and
Lemma~\ref{lem:weightAB} give
\[
\Fstar(m,n,Q,w)\ge N_k(H,W).
\]
By Lemma~\ref{lem:Nk},
\[
N_k(H,W)\ge\binom{HW}{k}(1-\delta),
\quad
\delta:=H\frac{\binom{(H-1)W}{k}}{\binom{HW}{k}}
+W\frac{\binom{H(W-1)}{k}}{\binom{HW}{k}}.
\]
If $k\le(H-1)W$, then
\[
\frac{\binom{(H-1)W}{k}}{\binom{HW}{k}}
=\prod_{i=0}^{k-1}\frac{(H-1)W-i}{HW-i}
\le\left(1-\frac1H\right)^k\le e^{-k/H}.
\]
The first inequality holds because each factor in the product is at most
$((H-1)W)/(HW)=1-1/H$.
If $k>(H-1)W$, the ratio is zero, so the same bound holds. Exchanging $H$ and $W$ gives
\[
\frac{\binom{H(W-1)}{k}}{\binom{HW}{k}}\le e^{-k/W}.
\]
Therefore
\[
\delta\le H e^{-k/H}+W e^{-k/W}.
\]
Since $k/(mn)\to\alpha$, for all sufficiently large $m,n$ we have
$k\ge\alpha mn/2$. Thus
\[
\delta\le m e^{-\alpha n/2}+n e^{-\alpha m/2}\longrightarrow0,
\]
where the last limit uses $m\asymp n$. Therefore
\[
\frac{\log_2\Fstar(m,n,Q,w)}{mn}
\ge\frac{\log_2\binom{HW}{k}}{mn}+\frac{\log_2(1-\delta)}{mn}
\]
for all sufficiently large $m,n$, when $\delta<1$. Lemma~\ref{lem:entropy}, together with
$k/HW\to\alpha$ and $HW/mn\to1$, shows that the right side tends to $h(\alpha)$. This proves
the matching lower bound.
\end{proof}

\begin{corollary}\label{cor:fixedweight}
Suppose $Q$ has a $1$-entry that is the unique $1$ in its row and in its column. Put
$H=m-s+1$ and $W=n-t+1$. Then for all $m\ge s$, $n\ge t$, and every integer $w$
satisfying $o(Q)\le w\le o(Q)-1+HW$,
\[
\Fstar(m,n,Q,w)\;\ge\;\binom{HW}{\,w-o(Q)+1\,}.
\]
\end{corollary}

\begin{proof}
Take the isolated $1$-entry as the anchor $(p,q)$, so $\mu_p=\nu_q=1$ and $c_0=o(Q)-1$ by
Lemma~\ref{lem:weightAB}. By Lemma~\ref{lem:special}(c), every nonzero block $B$ yields a strongly
$Q$-forcing matrix $A(B)$, which by Lemma~\ref{lem:weightAB} has weight $|B|_1+o(Q)-1$. The
integer $w-o(Q)+1$ is at least $1$. Choosing $B$ with this many ones makes $A(B)$ strongly
$Q$-forcing of weight $w$; there are
$\binom{HW}{w-o(Q)+1}$ such blocks, all nonzero, and $B\mapsto A(B)$ is injective. The displayed
bound follows.
\end{proof}

\begin{remark}\label{rem:notight}
For a \emph{fixed} weight $w>o(Q)$ the bound of Corollary~\ref{cor:fixedweight} is in general far
from sharp. For the identity pattern $Q=I_2$ (so $s=t=2$, $o(Q)=2$, and each $1$ is isolated),
Corollary~\ref{cor:fixedweight} gives
$\Fstar(m,n,I_2,3)\ge\binom{(m-1)(n-1)}2=\Theta(m^2n^2)$. On the other hand, for every
$r_1<r_2<r_3$ and $c_1<c_2<c_3$, the matrix whose only $1$'s are
$(r_1,c_1),(r_2,c_2),(r_3,c_3)$ is strongly $I_2$-forcing: the first two and last two entries
form copies witnessing all three. Hence
$\Fstar(m,n,I_2,3)\ge\binom m3\binom n3=\Omega(m^3n^3)$.
Theorem~\ref{thm:minweight} gives the exact count at minimum weight, and
Theorem~\ref{thm:entropy} gives the exponential growth rate at positive-density weight. The
construction does not determine the polynomial asymptotics of $\Fstar(m,n,Q,w)$ at a fixed
weight $w>o(Q)$.
\end{remark}

\section{Counting ordinary $Q$-forcing matrices}\label{sec:ordinary}

Ordinary $Q$-forcing is upward closed under the entrywise order. We prove that the canonical
minimum matrix $A_{m,n,Q}$ is its least element. Its zero entries can therefore be chosen
independently, giving $F(m,n,Q)=2^{mn-\mathfrak m(m,n,Q)}$.

\begin{definition}\label{def:ordforce}
Let $Q$ be an $s\times t$ pattern and let $A$ be an $m\times n$ $(0,1)$-matrix with $m\ge s$,
$n\ge t$. Say that $A$ is \emph{$Q$-forcing} if every $s\times t$ submatrix $B$ of $A$
satisfies $B\ge Q$ entrywise, that is, $B[a,b]=1$ at every position $(a,b)$ with $Q[a,b]=1$.
Let $F(m,n,Q)$ denote the number of $m\times n$ $Q$-forcing matrices.
\end{definition}

Equivalently, $A$ is $Q$-forcing if and only if every $s\times t$ submatrix can be turned into
$Q$ by changing some of its $1$-entries to $0$. This is the notion of Cao and
Tsai~\cite[Definition~1.1]{CT}, and as noted in the introduction it is incomparable with strong
forcing. We only use its upward closure below.

Define the $m\times n$ matrix $A_{m,n,Q}$ by declaring $A_{m,n,Q}[i,j]=1$ if and only if the cell
$(i,j)$ can occupy a $1$-position of $Q$ in some $s\times t$ submatrix; that is,
$A_{m,n,Q}[i,j]=1$ if and only if there are a $1$-entry $(a,b)$ of $Q$ (so $Q[a,b]=1$) and
increasing sequences $r_1<\dots<r_s$ in $[m]$ and $c_1<\dots<c_t$ in $[n]$ with $r_a=i$ and
$c_b=j$. The same eligible cells are obtained by considering only consecutive windows: if
$(i,j)$ can occupy role $(a,b)$ at all, then
$a\le i\le m-s+a$ and $b\le j\le n-t+b$, so the consecutive window beginning in row
$i-a+1$ and column $j-b+1$ places it in that role. Following \cite{CT}, let
$\mathfrak m(m,n,Q)$ denote the minimum of $|A|_1$ over all
$m\times n$ $Q$-forcing matrices. The following lemma is the part of
\cite[Lemma~2.1]{CT} needed here. We include the proof.

\begin{lemma}\label{lem:domination}
$A_{m,n,Q}$ is $Q$-forcing, and every $m\times n$ $Q$-forcing matrix $B$ satisfies
$B\ge A_{m,n,Q}$ entrywise. Hence $A_{m,n,Q}$ is the unique $Q$-forcing matrix of minimum weight,
and $\mathfrak m(m,n,Q)=|A_{m,n,Q}|_1$.
\end{lemma}

\begin{proof}
\emph{Domination.} Suppose $A_{m,n,Q}[i,j]=1$, witnessed by a $1$-entry $(a,b)$ of $Q$ and
sequences $r_\bullet,c_\bullet$ with $r_a=i$, $c_b=j$. If $B$ is $Q$-forcing, its $s\times t$
submatrix on rows $r_\bullet$ and columns $c_\bullet$ dominates $Q$, so its entry in position
$(a,b)$, namely $B[i,j]$, equals $1$. Thus $B\ge A_{m,n,Q}$ entrywise.

\emph{$A_{m,n,Q}$ is $Q$-forcing.} Let $A'$ be the submatrix of $A_{m,n,Q}$ on any rows
$r_1<\dots<r_s$ and columns $c_1<\dots<c_t$, and let $(a,b)$ satisfy $Q[a,b]=1$. The cell
$(r_a,c_b)$ occupies the $1$-position $(a,b)$ of $Q$ in this submatrix, so by definition
$A_{m,n,Q}[r_a,c_b]=1$, i.e.\ $A'[a,b]=1=Q[a,b]$. Hence $A'\ge Q$; as $A'$ was an arbitrary
$s\times t$ submatrix, $A_{m,n,Q}$ is $Q$-forcing.

\emph{Minimality and uniqueness.} For any $Q$-forcing $B$, domination gives $B\ge A_{m,n,Q}$, so
$|B|_1\ge|A_{m,n,Q}|_1$ with equality iff $B=A_{m,n,Q}$. Thus $A_{m,n,Q}$ is the unique
minimum-weight $Q$-forcing matrix and $\mathfrak m(m,n,Q)=|A_{m,n,Q}|_1$.
\end{proof}

\begin{theorem}\label{thm:ordcount}
Let $Q$ be an $s\times t$ pattern with at least one $1$-entry, and let $m\ge s$, $n\ge t$. Then
\[
F(m,n,Q)=2^{\,mn-\mathfrak m(m,n,Q)} .
\]
\end{theorem}

\begin{proof}
Lemma~\ref{lem:domination} shows that every $Q$-forcing matrix $B$ satisfies
$B\ge A_{m,n,Q}$. Conversely, suppose $B\ge A_{m,n,Q}$. Every $s\times t$ submatrix $A'$
of $A_{m,n,Q}$ dominates $Q$, and the corresponding submatrix $B'$ of $B$ satisfies
$B'\ge A'\ge Q$. Hence $B$ is $Q$-forcing. The $Q$-forcing matrices are therefore exactly
the matrices above $A_{m,n,Q}$. The $\mathfrak m(m,n,Q)$ $1$-entries of $A_{m,n,Q}$ are
fixed, while each of its other $mn-\mathfrak m(m,n,Q)$ entries can be chosen independently.
The formula follows.
\end{proof}

\begin{theorem}\label{thm:duality}
Let $Q$ be a nonzero $s\times t$ pattern, $m\ge s$, and $n\ge t$. Put
\[
\mathfrak m=\mathfrak m(m,n,Q),\qquad o=o(Q),\qquad
H=m-s+1,\qquad W=n-t+1,
\]
and let $r(Q)$ and $c(Q)$ be the numbers of nonzero rows and nonzero
columns of $Q$. Then
\[
\Fstar(m,n,Q,w)\le\binom{\mathfrak m}{w}
\]
for every integer $w\ge0$, and consequently
\[
\Fstar(m,n,Q)
\le 1+\sum_{w=o}^{\mathfrak m}\binom{\mathfrak m}{w}
=2^{\mathfrak m}-\sum_{w=1}^{o-1}\binom{\mathfrak m}{w}.
\]
Consequently,
\[
2^{\,mn-\mathfrak m+HW}
\le
F(m,n,Q)\,\Fstar(m,n,Q)
\le
2^{mn}.
\]
Equality in the left-hand inequality holds if and only if
\[
\bigl(H=1\text{ or }r(Q)=1\bigr)
\qquad\text{and}\qquad
\bigl(W=1\text{ or }c(Q)=1\bigr),
\]
whereas equality in the right-hand inequality holds if and only if
$Q$ has exactly one $1$-entry. In particular,
\[
2\le F(m,n,Q)\,\Fstar(m,n,Q),
\]
with equality if and only if $m=s$, $n=t$, and $Q=J_{s,t}$
(equivalently, $Q=J_{m,n}$).
\end{theorem}

\begin{proof}
Let $A$ be strongly $Q$-forcing and let $(i,j)$ be a $1$-entry of $A$. A witnessing copy
places $(i,j)$ at some $1$-position of $Q$. By the definition of the canonical matrix
$A_{m,n,Q}$, this means $A_{m,n,Q}[i,j]=1$. Hence every strongly $Q$-forcing matrix satisfies
$A\le A_{m,n,Q}$, so a matrix of weight $w$ has its support among the
$\mathfrak m=|A_{m,n,Q}|_1$ eligible cells. This proves the weightwise binomial bound.
Theorem~\ref{thm:minweight} shows that the zero matrix is the only strong matrix below weight
$o$, giving the displayed total bound.

Dropping the subtracted terms yields $\Fstar(m,n,Q)\le2^{\mathfrak m}$; multiplying by
Theorem~\ref{thm:ordcount} gives the upper product bound. If $o\ge2$, the term
$\binom{\mathfrak m}{1}$ is positive and the inequality is strict. If $o=1$, then $Q$ is a
singleton pattern. Put $H=m-s+1$ and $W=n-t+1$. Its eligible cells form the $H\times W$
rectangle in Lemma~\ref{lem:cornerchar}, so $\mathfrak m=HW$ and Theorem~\ref{thm:corner} gives
$\Fstar=2^{\mathfrak m}$; equality follows.

For the lower product bound, Theorems~\ref{thm:ordcount} and
\ref{thm:main} give
\[
F(m,n,Q)\,\Fstar(m,n,Q)
=
2^{\,mn-\mathfrak m}\Fstar(m,n,Q)
\ge
2^{\,mn-\mathfrak m+HW}.
\]
By Theorem~\ref{thm:equality}, equality holds here if and only if
\[
\bigl(H=1\text{ or }r(Q)=1\bigr)
\qquad\text{and}\qquad
\bigl(W=1\text{ or }c(Q)=1\bigr).
\]

Since $\mathfrak m\le mn$ and $HW\ge1$, the lower product bound implies
\[
F(m,n,Q)\,\Fstar(m,n,Q)\ge2.
\]
Suppose equality holds. Then
\[
mn-\mathfrak m+HW=1.
\]
Both $mn-\mathfrak m$ and $HW$ are nonnegative integers, and $HW\ge1$,
so necessarily $\mathfrak m=mn$ and $HW=1$. Since $H,W\ge1$, this gives
$H=W=1$, and hence $m=s$ and $n=t$. In this case the canonical minimum
matrix $A_{m,n,Q}$ is $Q$ itself, so $\mathfrak m=o(Q)$. Therefore
$o(Q)=mn$, which means that $Q=J_{m,n}$.

Conversely, if $m=s$, $n=t$, and $Q=J_{s,t}$, then $J_{m,n}$ is the
unique $Q$-forcing matrix, while the zero matrix and $J_{m,n}$ are the
only strongly $Q$-forcing matrices. Thus
\[
F(m,n,Q)\,\Fstar(m,n,Q)=1\cdot2=2.
\]
\end{proof}

The hypothesis that $Q$ has a $1$-entry is the standing convention of this paper and is not
needed for the formula: for the all-zero pattern every matrix is (vacuously) $Q$-forcing,
$\mathfrak m(m,n,Q)=0$, and Theorem~\ref{thm:ordcount} reads $F(m,n,Q)=2^{mn}$ correctly.

Ordinary forcing is upward closed, whereas strong forcing need not be. For example, $(0,0)$ is
strongly $(1,0)$-forcing, but $(0,1)$ is not. The square growth rates also differ.
Theorem~\ref{thm:main} gives $\log_2\Fstar(n,n,Q)=n^2(1-o(1))$, while
Theorem~\ref{thm:ordcount} and \cite{CT} give
\[
\log_2F(n,n,Q)=n^2-\mathfrak m(n,n,Q)=O(n).
\]
Thus the strong and ordinary square growth rates are $1$ and $0$, respectively.

\subsection{Identity patterns}

For $n\ge2k$, Cao and Tsai \cite{CT} proved that
$\mathfrak m(n,n,I_k)=n^2-k(k-1)$. We extend this formula to all $n\ge k$.

\begin{corollary}\label{cor:Ik}
Let $k\ge 1$ and $n\ge k$, and let $I_k$ be the $k\times k$ identity pattern. Using $0$-indexed
coordinates $\{0,\dots,n-1\}^2$, the minimum forcing matrix $A_{n,n,I_k}$ has
\[
A_{n,n,I_k}[i,j]=0\iff |i-j|\ge n-k+1,
\]
so its $0$-cells form the two triangular corners $\{j-i\ge n-k+1\}$ and $\{i-j\ge n-k+1\}$,
each of size $1+2+\dots+(k-1)=\binom k2$. Consequently $\mathfrak m(n,n,I_k)=n^2-k(k-1)$ and
\[
F(n,n,I_k)=2^{\,k(k-1)}\qquad\text{for every }n\ge k.
\]
\end{corollary}

\begin{proof}
By the definition of $A_{n,n,I_k}$, its entry $(i,j)$ equals $1$ if and only if there is an
integer $a\in\{0,\dots,k-1\}$ such that
\[
\max(i,j)-(n-k)\ \le\ a\ \le\ \min(i,j).
\]
The lower endpoint is at most $k-1$, and the upper endpoint is at least $0$. Thus such an
integer $a$ exists if and only if
\[
\max(i,j)-(n-k)\le\min(i,j),
\]
which is equivalent to $|i-j|\le n-k$. This proves the stated description of the zero cells.
For each $d\in\{n-k+1,\dots,n-1\}$, there are $n-d$ cells with $j-i=d$. Hence each triangular
corner has
\[
\sum_{d=n-k+1}^{n-1}(n-d)=1+2+\dots+(k-1)=\binom{k}{2}
\]
cells. The corners are disjoint because $n-k+1\ge1$, so
\[
n^2-\mathfrak m(n,n,I_k)=|Z|=2\binom k2=k(k-1).
\]
Theorem~\ref{thm:ordcount} now gives $F(n,n,I_k)=2^{k(k-1)}$ for every $n\ge k$.
\end{proof}

\section{Concluding questions}

Theorem~\ref{thm:main} and Theorem~\ref{thm:equality} determine the smallest possible strong
count, and all its equality cases, among nonzero patterns of fixed dimensions. Several
enumerative problems remain. Which nonzero $s\times t$ patterns maximize $\Fstar(m,n,Q)$, and what
further pattern families admit exact formulas? At fixed weight $w>o(Q)$, the polynomial order of
$\Fstar(m,n,Q,w)$ is open in general; Remark~\ref{rem:notight} shows that the full-scaffold
bound need not have the correct order. How far can the comparability hypothesis in
Theorem~\ref{thm:entropy} be weakened? It remains to characterize equality in the
weightwise canonical-support bound of Theorem~\ref{thm:duality}.

\section*{Declaration of generative AI and AI-assisted technologies}

During the preparation of this draft, JG used Codex with GPT-5.6 Sol Ultra to assist with manuscript editing. The authors reviewed and edited the content as needed and take full responsibility for the content of the article.

\end{document}